\documentclass[a4paper]{amsart}

\usepackage{amsthm}
\usepackage{amsmath}
\usepackage{amssymb}
\usepackage{latexsym}
\usepackage{enumerate}
\usepackage{graphicx}
\usepackage[utf8]{inputenc}
\usepackage{epsfig} 
\usepackage{pgf}
\usepackage{tikz}
\usepackage{float}
\usepackage{dsfont}
\usepackage{times}

\newtheorem{theoremA}{Theorem}

\renewcommand{\thetheoremName}

\newtheorem{proposition[[]]}[theoremName]{Proposition G}

\newtheorem{theorem}{Theorem}[section]
\newtheorem{lemma}[theorem]{Lemma}

\newtheorem{proposition}[theorem]{Proposition}

\theoremstyle{definition}
\newtheorem{definition}[theorem]{Definition}

\newtheorem{remark}{Remark}

\numberwithin{equation}{section}

\newcommand{\D}{\operatorname{D}}

\newcommand{\dist}{\operatorname{dist}}
\newcommand{\Vol}{\operatorname{Vol}}

\newcommand{\Sup}{\operatorname{Sup}}

\newcommand{\kan}{\mathbb{K}^{n}(b)}

\newcommand{\erre}{\mathbb{R}}

\newcommand{\Han}{\mathbb{H}^n(b)}

\newcommand{\gr}{\operatorname{\nabla}}
\begin{document}

\title[Extrinsic isoperimetry and compactification]
{Extrinsic isoperimetry and compactification of minimal surfaces in Euclidean and Hyperbolic spaces}
\author[V. Gimeno]{Vicent Gimeno$^{\#}$}
\address{Departament de Matem\`{a}tiques-INIT, Universitat Jaume I, Castell\'o,
Spain.}
\email{gimenov@guest.uji.es}

\author[V. Palmer]{Vicente Palmer*}
\address{Departament de Matem\`{a}tiques-INIT, Universitat Jaume I, Castell\'o,
Spain.} \email{palmer@mat.uji.es}
\thanks{$^{\#}$ Supported by the
Fundaci\'o Caixa Castell\'o-Bancaixa Grants P1.1B2006-34 and P1.1B2009-14\\
\indent * Supported by MICINN grant No. MTM2010-21206-C02-02.}



\begin{abstract}
We study the topology of (properly) immersed complete minimal surfaces $P^2$ in Hyperbolic and Euclidean spaces  which have finite total extrinsic curvature, using  some isoperimetric inequalities satisfied by the extrinsic balls in these surfaces, (see \cite{Pa}). We present an alternative and  partially unified proof of the Chern-Osserman inequality satisfied by these minimal surfaces, (in $\erre^n$ and in $\Han$), based in the isoperimetric analysis above alluded.  Finally, we show a Chern-Osserman type equality attained by complete minimal surfaces in the Hyperbolic space with finite total extrinsic curvature.
\end{abstract}

\maketitle

\section{Introduction}\label{secIntro}
Let us consider $P^2$ be a complete and minimal surface immersed in $\erre^n$ and with finite total curvature $\int_P K^Pd\sigma <\infty$, being $K^P$ the Gauss curvature of the surface. Then we have the following equality (resp. inequality), known as the {\em Chern-Osserman formula}, (see \cite{A1}, \cite{ChOss} and \cite{JM}):

\begin{equation}\label{ChernOssEq}
-\chi(P)= \frac{1}{4\pi}\int_P\Vert B^P\Vert^2 d\sigma -\Sup_{r}\frac{\Vol(P^2\cap B^{0,n}_r)}{\Vol(B^{0,2}_r)}\leq \frac{1}{4\pi}\int_P\Vert B^P\Vert^2 d\sigma -k(P)
\end{equation}
where  $\chi(P)$  is the Euler characterisitic of $P$, $k$ is its number of ends, $B^P$ is the second fundamental foorm of $P$ in $\erre^n$ and $B^{b,n}_r$ denotes the geodesic $r$-ball in the simply connected real space form $\kan$.

To have finite total scalar (extrinsic) curvature $\int_P \Vert B^P\Vert^2 d\sigma <\infty$ is equivalent  to the finiteness of the total Gaussian curvature (the original assumption in \cite{ChOss}) when the surface is minimal and immersed in $\erre^n$.  From this point of view, it is natural to wonder if it is possible to stablish a Chern-Osserman inequality (or equality) for complete minimal surfaces with finite total extrinsic curvature (properly) immersed in the hyperbolic space. This question has been addressed by Q. Chen and Y. Cheng in the papers \cite{Ch2} and \cite{Che3}. They proved, for a complete minimal surface $P^2$ (properly) immersed in $\Han$ and such that $\int_P \Vert B^P\Vert d\sigma <\infty$, that $  \Sup_{r}\frac{\Vol(P^2\cap B^{-1,n}_r)}{\Vol(B^{-1,2}_r)}< \infty$ and the following version of the Chern-Osserman Inequality, in terms of the volume growth of the extrinsic balls:
\begin{equation}\label{ChernOssHyp}
 -\chi(P)\leq \frac{1}{4\pi}\int_P\Vert B^P\Vert^2 d\sigma -\Sup_{r}\frac{\Vol(P^2\cap B^{-1,n}_r)}{\Vol(B^{-1,2}_r)}
\end{equation}

The proofs given by these authors are different for those for the Euclidean case, and rely heavily on the properties of the hyperbolic functions.

We present in this paper a partial unification of the proof of the Chern-Osserman inequality (in terms of the volume growth) for complete minimal surfaces with finite total extrinsic curvature immersed in Euclidean or Hyperbolic spaces. This partial unification is based in obtaining estimates for the Euler characteristic of the extrinsic balls (given in Lemma \ref{gradient}, and Proposition \ref{Mainth}) and in the isoperimetric inequality for the extrinsic balls given in Theorem 1.1 in \cite{Pa}. These results are based, in its turn, on the divergence Theorem and the Hessian and Laplacian comparison theory of restricted distance function, (see \cite{GW}, \cite{HP} and \cite{Pa3}) which involves bounds on the mean curvature of the submanifold.

We have proved the following Chern-Osserman inequality, which encompasses inequalities (\ref{ChernOssEq}) and (\ref{ChernOssHyp}):

\begin{theoremA}\label{ChernOss1} Let $P^2$ be an complete minimal surface immersed in a simply connected real space form with constant sectional curvature $b \leq 0$, $\kan$. Let us suppose that $\int_P\Vert B^P \Vert^2 d\sigma<\infty$. Then

\begin{enumerate}
\item $P$ has finite topological type.
\item $\Sup_{t>0}(\frac{\Vol(D_t)}{\Vol(B^{b,2}_t)})<\infty$
\item $-\chi(P)\leq \frac{\int_P \Vert B^P \Vert ^2}{4\pi}-\Sup_{t>0}\frac{\Vol(D_t)}{\Vol(B_t^{b,2})}$
\end{enumerate}
where $\chi(P)$ is the Euler characteristic of $P$. 
\end{theoremA}

  Although with this approach we are not able to state equality (\ref{ChernOssEq}) in the Euclidean setting, we shall prove in Theorem \ref{ChernEq} the following Chern-Osserman type equality for cmi surfaces in the Hyperbolic space:

\begin{theoremA} \label{ChernEq}
Let $P^2$ be  a complete immersed minimal surface in $\Han$. Let us suppose that $\int_P \Vert B^P\Vert^2 d\sigma < \infty$. Then
\begin{equation}\label{Equality}
-\chi(P)=\frac{1}{4\pi}\int_P \Vert B^P \Vert ^ 2 d\sigma-\Sup_{t>0}\frac{\Vol(D_t)}{\Vol(B_t^{b,2})}-\frac{1}{2\pi}G_b(P)
\end{equation}
where $G_b(P)$ is a nonnegative and finite quantity which do not depends on the exhaustion by extrinsic balls $\{D_t\}_{t>0}$ of $P$ and is given by
\begin{equation}
\begin{aligned}
G_b(P)&:= \lim_{t \to \infty}\left(h_b(t)\Vol(B^{b,2}_t)(\frac{(\Vol(D_t))}{\Vol(B^{b,2}_t)})' \right.\\&\left.+\int_{\partial D_t} \langle B^P(e,e),\frac{\gr^\perp r}{\Vert \gr^P r \Vert}\rangle d\sigma_t\right)
\end{aligned}
\end{equation}
\end{theoremA}
 
\subsection{Outline}
The outline of the paper is following.
In Section \S.2 we present the basic facts about the Hessian comparison theory of restricted distance function we are going to use, obtaining as a corollary the compactification of cmi surfaces in $\kan$ with finite total extrinsic curvature, (Corollary \ref{CorDifeo}). Section \S.3 is devoted to the unified proof of the Chern-Osserman inequality for complete minimal surfaces with finite total extrinsic curvature immersed in Euclidean and Hyperbolic spaces (Theorem \ref{ChernOss1}), and in Section \S.4 it is proved a Chern-Osserman type equality satisfied by the cmi surfaces in $\Han$ (Theorem \ref{ChernEq}).

\section{Preliminaires} \label{Prelim}  


\subsection{The extrinsic distance}
We assume
throughout the paper that $P^{2}$ is a complete,  non-compact, immersed,
$2$-dimensional submanifold in a simply connected real space form of non-positive constant sectional curvature $\kan$, ($\kan=\erre^n$ when $b=0$ and $\kan=\Han$ when $b<0$) . All the points in these manifolds are poles. Recall that a pole
is a point $o$ such that the exponential map
$$\exp_{o}\colon T_{o}N^{n} \to N^{n}$$ is a
diffeomorphism. For every $x \in N^{n}\setminus \{o\}$ we
define $r_o(x) = \dist_{N}(o, x)$, and this
distance is realized by the length of a unique
geodesic from $o$ to $x$, which is the {\it
radial geodesic from $o$}. We also denote by $r$
the restriction $r\vert_P: P\to \erre_{+} \cup
\{0\}$. This restriction is called the
{\em{extrinsic distance function}} from $o$ in
$P^m$. The gradients of $r$ in $N$ and $P$ are
denoted by $\gr^N r$ and $\gr^P r$,
respectively. Let us remark that $\gr^P r(x)$
is just the tangential component in $P$ of
$\gr^N r(x)$, for all $x\in S$. Then we have
the following basic relation:
\begin{equation}\label{radiality}
\nabla^N r = \gr^P r +(\gr^N r)^\bot ,
\end{equation}
where $(\gr^N r)^\bot(x)=\nabla^\bot r(x)$ is perpendicular to
$T_{x}P$ for all $x\in P$.

On the other hand, we should recall that all immersed surfaces $P$ in the real space forms of non-positive constant sectional curvature $N^n=\kan$ which satisfies $\int_P\Vert B^P\Vert^2 d\sigma <\infty$ are properly immersed (see \cite{A1}, \cite{MS} and \cite{O}). Therefore, we can omit the hypothesis about the properness of the immersion when we assume that $\int_P\Vert B^P\Vert^2 d\sigma <\infty$.

\begin{definition}\label{ExtBall}
Given a connected and complete
surface $P^2$ properly immersed in a manifold $N^n$ with a pole $o \in N$, we
denote
the {\em{extrinsic metric balls}} of radius $t >0$ and center $o \in N$ by
$D_t(o)$. They are defined as the intersection
$$
D_t(o)=B^N_{t}(o) \cap P =\{x\in P \colon r(x)< t\},
$$
where $B^N_{t}(o)$ denotes the open geodesic ball
of radius $R$ centered at the pole $o$ in
$N^{n}$.
\end{definition}

\begin{remark}\label{theRemk0}
We want to point out that the extrinsic domains $D_t(o)$
are precompact sets, (because we assume in the definition above that  the submanifold $P$ is properly immersed), 
with boundary $\partial D_t(o)$ being a immersed curve in $P$.  The generical smoothness of
$\partial D_{t}(o)$ follows from the following considerations: 
the distance function $r$ is smooth in $\kan \setminus \{o\}$ 
since $\kan$  to possess a pole $o\in \kan$, ($b\leq 0$). Hence
the restriction $r\vert_P$ is smooth in $P$ and consequently the
radii $t$ that produce smooth boundaries
$\partial D_{t}(o)$ are dense in $\mathbb{R}$ by
Sard's theorem and the Regular Level Set Theorem.

\end{remark}

\begin{remark}\label{theRemk1}
When the submanifold considered is totally geodesic, namely, when $P$ is a Hyperbolic or an Euclidean subespace of the ambient real space form, the extrinsic balls become geodesic balls, and its boundary is the distance sphere. We recall here that the mean curvature of the geodesic sphere in the real space form $\kan$, 'pointed inward' is (see \cite{Pa}):
$$
h_b(t)=\left\{
\begin{array}{l}
\sqrt{b}\cot\sqrt{b}t\,\,\text{  if }\,\,b>0\\
1/t\,\,\text{  if }\,\, b=0\\
\sqrt{-b}\coth\sqrt{-b}t\,\, \text{  if }\,\, b<0
\end{array}\right.
$$
\end{remark}


\subsection{Hessian comparison analysis of the extrinsic distance}

Let us consider now $D_t$ an extrinsic ball in a complete and properly  immersed minimal surface $P$ in the real space form $\kan$ with $b\leq 0$. 
We are going to apply Gauss-Bonnet formula to the curve $\partial D_t$. To do that, we need to compute its geodesic curvature in the following 
\begin{proposition}\label{geodCurv}
Given $\partial D_t$ the smooth closed curves in $P$, 

\begin{equation}\label{equ:kg}
k_g^{\partial D_t}=\frac{h_b(t)}{\Vert \gr^Pr\Vert}+\langle B^P(e,e),\frac{\gr^{\bot}r}{\Vert \gr^Pr\Vert}\rangle
\end{equation}
\end{proposition}
\begin{proof}
Let $\{e, \nu\} \subset TP$ be an orthonormal frame along the curve $\partial D_t$, where $e$ is the unit tangent vector to $\partial D_t$ and $\nu=\frac{\gr^P r}{\Vert \gr^P \Vert}$ is the unit normal to $\partial D_t$ in $P$, pointed outward.

From the definition of geodesic curvature of the
extrinsic boundaries $\partial D_{t}$, we have
\begin{equation}
k_{g}^{t}=-\langle\nabla_{e}^{P}e,\frac{\nabla^{P}r}{\Vert\nabla^{P}r\Vert}\rangle
\end{equation}
Then, having on account the definition of Hessian
\[
Hess^{P}r(e,e)=\langle \nabla^{P}\nabla^{P}r,e\rangle
\]
and the fact that $\nabla^{P}r$ and $e$ are orthogonal,
\begin{equation}
k_{g}^{t}=\frac{1}{\Vert\nabla^{P}r\Vert}Hess^{P}r(e,e)
\end{equation}

But, given $X\in
T_{q}P$ unitary, (see \cite{HP} and \cite{Pa3} for detailed computations):
\begin{equation}  \label{eqHess1}
\operatorname{Hess}^P(r)(X,X)\, = \, h_b(r)\left(\, 1-  \langle\, X,\nabla^{\kan} r\,\rangle^{2} \, \right) +  \langle \, \nabla^{\kan}r, \, B^P(X,X) \, \rangle \,  
\end{equation}
where $B^P$ is the second fundamental form of $P$ in $N$. 
\bigskip
Applying at this point equation (\ref{eqHess1}):

\begin{equation}
k_{g}^{t}=\frac{1}{\Vert\nabla^{P}r\Vert}\{h_{b}(r)+\langle \nabla
^{\perp}r,B^{P}(e,e)\rangle\}
\end{equation}
\end{proof}

Now, we consider $\{D_t\}_{t>0}$ an exhaustion of $P$ by extrinsic balls. Recall than an exhaustion of the submanifold $P$ is a sequence of subsets $\{D_t\subseteq P\}_{t>0}$ such that:
\begin{itemize}
\item $D_t \subseteq D_s$ when $s\geq t$
\item $\cup_{t>0} D_t=P$
\end{itemize}

Using the equality (\ref{equ:kg}) for the geodesic curvature of the extrinsic curves we have the following result

\begin{theorem}\label{CorDifeo}
Let $P^2$ be an complete minimal surface immersed in a simply connected real space form with constant sectional curvature $b \leq 0$, $\kan$. Let us suppose that $\int_P\Vert B^P \Vert^2 d\sigma<\infty$. Then 

(i) $P$ is diffeomorphic to a compact surface $P^*$ punctured at a finite number of points. 
\medskip

(ii) For all sufficiently large $t>R_0>0$, $\chi(P)=\chi(D_t)$ and hence, given $\{D_t\}_{t>0}$ an exhaustion of $P$ by extrinsic balls, 
$$ \chi(P)= \lim_{t \to \infty} \chi(D_t)$$

\end{theorem}
\begin{proof} 
 Let us consider $\{D_t\}_{t>0}$ an exhaustion of $P$ by extrinsic balls, centered at the pole $o \in \kan$. 
We apply Lemma \ref{equ:kg} to the smooth curves $\partial D_t$: As 
$$-\Vert B^P\Vert \leq \langle B^P(e,e),\gr^{\bot}r\rangle  \leq \Vert B^P\Vert $$

we have, on the points of the curve $ q \in \partial D_t$,
\begin{equation}\label{grg0}
\begin{aligned}
\Vert \gr^Pr\Vert(q) \cdot k_g^{\partial D_t}(q)&=h_b(r_o(q))+\langle B^P(e,e),\gr^{\bot}r\rangle(q)\,\\&\geq h_b(r_o(q))-\Vert B^P \Vert(q)
\end{aligned}
\end{equation}
Using now Proposition 2.2 in \cite{A1}, when $P^2$ is a cmi in $\erre^n$ or Lemma 3.1 in \cite{O},  when $P^2$ is a cmi in $\Han$, we know that $\Vert B^P \Vert(q)$ goes uniformly to $0$ as $t= r_o(q) \to \infty$. Hence,  for all the points $q \in \partial D_t$ and for sufficiently large $t$,

\begin{equation}\label{grg0}
\Vert \gr^Pr\Vert(q) \cdot k_g^{\partial D_t}(q)\ >0
\end{equation}
Hence, $\Vert \gr^Pr\Vert >0$ in $\partial D_t$, for all sufficiently large $t$. Fixing a  sufficienty large  radius $R_0$, we can conclude that the extrinsic distance $r_o$ has no critical points in $P\setminus \D_{R_0}$.

The above inequality implies that for this sufficienty large fixed radius $R_0$, there is a diffeomorphism

$$
\Phi: P\setminus \D_{R_0} \to \partial D_{R_0} \times [0,\infty[
$$

In particular, $P$ has only finitely many ends, each of finite topological type.

To proof this we apply Theorem 3.1 in \cite{Mi}, concluding that, as the extrinsic annuli $A_{R_0,R}(o)= D_R(o) \setminus D_{R_0}(o)$ contains no critical points of the extrinsic distance function $r_o: P \longrightarrow \erre^+$  because inequality (\ref{grg0}), then $D_R(o)$ is diffeomorphic to $D_{R_0}(o)$ for all $R \geq R_0$.

The above diffeomorfism implies that we can construct $P$ from $D_{R_0}$ ($R_0$ big enough) attaching annulis and that $\chi(P\setminus D_{t})=0$ when $t\geq R_0$.
Then, for all $t>R_0$,

$$
\chi(P)=\chi(D_{t}\cup (P\setminus D_{t}))=\chi(D_{t}) 
$$
\end{proof}

\section{Proof of Theorem \ref{ChernOss1}}\label{Cor}

We begin with the following results which are the common ingredient of the proof, both for the Euclidean and Hyperbolic cases :

\begin{lemma}\label{gradient}
Let $P^2\subset \kan$ be a  surface properly immersed in a real space form with curvature $b\leq 0$, let $D_t$ be an extrinsic disc in $P$ of radius $t >0$ and let $\partial D_t$ the extrinsic circle. Then: 
\begin{equation}\label{intgrad}
\int_{\partial D_t}\frac{||\nabla^\bot r||^2}{||\nabla^P r||}d\sigma_t\leq \int_{\partial D_t}\frac{1}{||\nabla^P r||}-h_b(t)\Vol(D_t)d\sigma_t
\end{equation}
\end{lemma}
\begin{proof}
Tracing equality (\ref{eqHess1}) we obtain the following expression for the Laplacian of the extrinsic distance in this context:

 \begin{equation} \label{eqLap1}
\Delta^{P}(r) \,= \, (m -\Vert \nabla^{P} r \Vert^{2})h_b(r)
 + m 
\langle \, \nabla^{N}r, \, H_{P}  \, \rangle  \quad ,
\end{equation}
where $H_{P}$ denotes the mean curvature vector
of $P$ in $N$ and $h_b(r)$ is the mean curvature of the geodesic $r$-spheres in $\kan$.
Applying divergence theorem  we have
\begin{equation}
\begin{aligned}
&\int_{\partial D_t}\frac{||\nabla^\bot r||^2}{||\nabla^P r||}d\sigma_t= \int_{\partial D_t}\frac{1}{||\nabla^P r||}d\sigma_t-\int_{\partial D_t}||\nabla^P r||d\sigma_t=\int_{\partial D_t}\frac{1}{||\nabla^P r||}d\sigma_t\\&-\int_{D_t}\Delta^P rd\sigma=\int_{\partial D_t}\frac{1}{||\nabla^P r||}d\sigma_t-\int_{D_t}(2-||\nabla^P r||^2)h_b(r)d\sigma\\&\leq \int_{\partial D_t}\frac{1}{||\nabla^P r||}d\sigma_t-
\int_{D_t}h_b(r)d\sigma\leq \int_{\partial D_t}\frac{1}{||\nabla^P r||}d\sigma_t-h_b(t)\Vol(D_t)
\end{aligned}
\end{equation}

\end{proof}

\begin{proposition}
\label{Mainth}
Let $P^2\subset \kan$ be a complete minimal surface properly immersed in a real space form with curvature $b\leq 0$, let $D_t$ be an extrinsic disc in $P$ of radius $t >0$ and let $\partial D_t$ be its boundary. Then: 
\begin{equation}
\begin{aligned}
-2\pi\chi &(D_t) +(b+\frac{f_{b,\alpha}^2(t)h_b(t)}{2 })\Vol(D_t)\\&+(h_b(t)-\frac{f_{b,\alpha}^2(t)}{2})\int_{\partial D_t} \frac{1}{\Vert \nabla^P r\Vert} d\sigma_t\leq \frac{1}{2}R(t)+\frac{1}{2 f_{b,\alpha}^2(t)}R'(t)
\end{aligned}
\end{equation}
\noindent where $R(t)=\int_{D_t}\Vert B^P\Vert^2 d\sigma$,  $\Vert B^P\Vert$ is the norm of the second fundamental form of $P$ in $\kan$, $\chi(D_t)$ is the Euler's characterisc of $D_t$ and, given  $\alpha \in ]0,2[$\, ,
$$
f_{b,\alpha}^2(t)=\alpha h_b(t)
$$

\end{proposition}
\begin{proof}
Integrating along $\partial D_t$ equation (\ref{equ:kg}) and using Gauss-Bonnet theorem and co-area formula, (see \cite{S}), we obtain
\begin{equation}
\begin{aligned}
2&\pi\chi (D_t) -\int_{D_t}K^P d\sigma=\\&h_b(t)\int_{\partial D_t}\frac{1}{\Vert \gr^Pr\Vert}d\sigma_t+
\int_{\partial D_t}\langle B^P(e,e),\frac{\gr^{\bot}r}{\Vert \gr^Pr\Vert}\rangle d\sigma_t\end{aligned}
\end{equation}
where we denote as $K^P$ the Gauss curvature of $P$.

But , on $\partial D_t$, 

$$-\Vert B^P\Vert \frac{\Vert \gr^{\bot}r\Vert}{\Vert \gr^Pr\Vert} \leq \langle B^P(e,e),\frac{\gr^{\bot}r}{\Vert \gr^Pr\Vert}\rangle \leq \Vert B^P\Vert \frac{\Vert \gr^{\bot}r\Vert}{\Vert \gr^Pr\Vert}$$

\noindent so, as $f_{b,\alpha}(t) \geq 0\,\forall t>0$,  having into account the inequality among the arithmetic and geometric mean and applying co-area formula:

\begin{equation}\label{bigineq}
\begin{aligned}
2&\pi\chi (D_t) -\int_{D_t}K^Pd\sigma=h_b(t)\int_{\partial D_t}\frac{1}{\Vert \gr^Pr\Vert}d\sigma_t\\&+\int_{\partial D_t}\langle B^P(e,e),\frac{\gr^{\bot}r}{\Vert \gr^Pr\Vert}\rangle d\sigma_t \, \geq\, h_b(t) \int_{\partial D_t}\frac{1}{\Vert \gr^Pr\Vert}d\sigma_t\\&-\frac{1}{2}\int_{\partial D_t}\frac{\Vert B^P\Vert ^2}{f_{b,\alpha}^2(r)\Vert \gr^Pr\Vert}d\sigma_t-\frac{1}{2} \int_{\partial D_t}\frac{f_{b,\alpha}^2(r)\Vert \gr^{\bot}r\Vert ^2}{\Vert \gr^Pr\Vert}d\sigma_t\\&\geq h_b(t) \int_{\partial D_t}\frac{1}{\Vert \gr^Pr\Vert}d\sigma_t-\frac{1}{2 f_{b,\alpha}^2(t)}R'(t)-\frac{f_{b,\alpha}^2(t)}{2}\int_{\partial D_t}\frac{\Vert \gr^{\bot}r\Vert ^2}{\Vert \gr^Pr\Vert}d\sigma_t
\end{aligned}
\end{equation}

Then, using inequality (\ref{intgrad}) of Lemma  \ref{gradient}  in the last member of the inequalities (\ref{bigineq}) and  applying Gauss equation for minimal surfaces in the real space forms $\kan$, we have

\begin{equation}
\begin{aligned}
2\pi\chi &(D_t) -b\Vol(D_t)+\frac{1}{2}R(t) \geq (h_b(t)-\frac{f_{b,\alpha}^2(t)}{2}) \int_{\partial D_t}\frac{1}{\Vert \gr^Pr\Vert}d\sigma_t\\ &-\frac{1}{2 f_{b,\alpha}^2(t)}R'(t)+\frac{f_{b,\alpha}^2(t)h_b(t)}{2 }\Vol(D_t)
\end{aligned}
\end{equation}
and hence

\begin{equation}
\begin{aligned}
-2\pi\chi& (D_t) +(b+\frac{f_{b,\alpha}^2(t)h_b(t)}{2 })\Vol(D_t)\\&+(h_b(t)-\frac{f_{b,\alpha}^2(t)}{2})\int_{\partial D_t} \frac{1}{\Vert \nabla^P r\Vert} \leq \frac{1}{2}R(t)+\frac{1}{2 f_{b,\alpha}^2(t)}R'(t)
\end{aligned}
\end{equation}
\end{proof}

We are going to divide the proof in two cases: the  {\em Case I}, where the ambient space is the Hyperbolic space $\Han$, and  the  {\em Case II} where the ambient space is the Euclidean space $\erre^n$. 
\subsection*{Case I }

Let us consider $P$ (properly) immersed in $\Han$.
Let $\{D_t\}_{t>0}$ be an exhaustion of $P$ by extrinsic balls. 
Using co-area formula, we know that 
\begin{equation}\label{coareaIneq}
\frac{d}{dt} \Vol(D_t)=\int_{\partial D_t} \frac{1}{\Vert \nabla^P r\Vert}d\sigma_t
\end{equation} 

Hence, applying Proposition \ref{Mainth} we have

\begin{equation}\label{mainthcor1}
\begin{aligned}
-2&\pi\chi (D_t) +(b+\frac{f_{b,\alpha}^2(t)h_b(t)}{2 })\Vol(D_t)\\&+(h_b(t)-\frac{f_{b,\alpha}^2(t)}{2})\frac{d}{dt} \Vol(D_t)\leq \frac{1}{2}R(t)+\frac{1}{2 f_{b,\alpha}^2(t)}R'(t)
\end{aligned}
\end{equation}
 On the other hand, from \ref{coareaIneq},  $\frac{d}{dt} \Vol(D_t) \geq \Vol(\partial D_t)$. Therefore, using inequality (\ref{mainthcor1}) we obtain

\begin{equation}
\begin{aligned}
&-2\pi\chi(D_t)\\& +\Vol(D_t)\left[(b+\frac{f_{b,\alpha}^2(t)h_b(t)}{2 })+(h_b(t)-\frac{f_{b,\alpha}^2(t)}{2}) \frac{\Vol(\partial D_t)}{\Vol(D_t)}\right] \\& \leq \frac{1}{2}R(t)+\frac{1}{2 f_{b,\alpha}^2(t)}R'(t)
\end{aligned}
\end{equation}

\noindent Applying isoperimetric inequality in \cite{Pa}, (Theorem 1.1), we have

\begin{equation}
\begin{aligned}
&-2\pi\chi (D_t) \\&+\Vol(D_t)\left[(b+\frac{f_{b,\alpha}^2(t)h_b(t)}{2 })+(h_b(t)-\frac{f_{b,\alpha}^2(t)}{2}) \frac{\Vol(S_t^{b,1})}{\Vol(B_t^{b,2})}\right] \\&\leq\frac{1}{2}R(t)+\frac{1}{2 f_{b,\alpha}^2(t)}R'(t)
\end{aligned}
\end{equation}
Hence, using the fact that 
$$b\Vol(B^{b,2}_t)+h_b(t)\Vol(S^{b,1}_t=2\pi\,\,\,\forall t>0$$
we obtain, with some computations
\begin{equation}
\begin{split}
-2\pi\chi (D_t) +\frac{\Vol(D_t)}{\Vol(B_t^{b,2})}&\left[2\pi-2\pi\frac{f_{b,\alpha}^2(t)}{2 }\frac{\Vol(B^{b,2}_t)}{\Vol(S_t^{b,1})}\right] \\ \leq\frac{1}{2}R(t)+\frac{1}{2 f_{b,\alpha}^2(t)}R'(t)
\end{split}
\end{equation}
Therefore, for all $t >0$,
\begin{equation}\label{substituida}
\begin{split}
\frac{\Vol(D_t)}{\Vol(B_t^{b,2})}\left(1-\frac{\alpha h_b(t)}{2}\frac{\Vol(B_t^{b,2})}{\Vol(S_t^{b,1})}\right)&-\chi(D_t)\\\leq \frac{R(t)}{4\pi}+\frac{R'(t)}{4\pi \alpha h_b(t)}
\end{split}
\end{equation}
As $ \frac{||B^P||^2}{h_b(t)} \leq \frac{1}{\sqrt{-b}}||B^P||^2$, then $\int_P ||B^P||^2 d\sigma < \infty$ implies $\int_P \frac{||B^P||^2}{h_b(t)}d\sigma<\infty$. Hence, by co-area formula:

 \begin{equation}
 \int_0^\infty \left( \int_{\partial D_t}\frac{||B^P||^2}{||\gr^P r|| h_b(r)} \right)dt=\int_0^\infty \left(\frac{R'(t)}{h_b(t)}\right)dt<\infty
 \end{equation}
 
Therefore,  there is a monotone increasing (sub)sequence $\{t_i\}_{i=1}^\infty$ tending to infinity, (namely, $t_i \to \infty$ when $i \to \infty$), such that $\frac{R'(t_i)}{h_b(t_i)}\rightarrow 0$ when $i\to \infty$. 

Let us consider the exhaustion of $P$ by these extrinsic balls, namely, $\{D_{t_i}\}_{i=1}^\infty$. Then
we have, replacing $t$ for $t_i$ and taking limits when $i \to \infty$ in inequality (\ref{substituida}) and applying Theorem \ref{CorDifeo} (ii),

\begin{equation}\label{Limsub}
\begin{split}
\Sup_{i}&\frac{\Vol(D_{t_i})}{\Vol(B_{t_i}^{b,2})}\left(1-\frac{\alpha }{2}\right)-\chi(P)\\ &\leq \lim_{i \to \infty}\frac{R(t_i)}{4\pi}=\frac{1}{4\pi}\int_P\Vert B^P\Vert^2 d\sigma < \infty
\end{split}
\end{equation}
 for all $\alpha$ such that $0<\alpha<2$.
 
Hence, as $\frac{\Vol(D_{t})}{\Vol(B_{t}^{b,2})}$ is a continuous non decreasing function of $t$,  we can conclude that $\Sup_{t>0}\frac{\Vol(D_t)}{\Vol(B_t^{b,2})} < \infty$ and $-\chi(P)<\infty$.

Then, letting $\alpha$ tend to $0$ in (\ref{Limsub}), we get, for all $t>0$:
\begin{equation}
\begin{aligned}
\Sup_{t>0}\frac{\Vol(D_t)}{\Vol(B_t^{b,2})}-\chi(P)\leq \frac{\int_P \Vert B^P \Vert ^2}{4\pi} 
\end{aligned} 
\end{equation}

\subsection*{Case II}

Let us consider $P$ immersed in $\erre^n$.
We consider, as in the proof above, an exhaustion of $P$ by extrinsic balls, $\{D_t\}_{t>0}$, but now, and following \cite{A1}, these extrinsic balls will be centered at the origin $0 \in \erre^n$, which we assume, without loss of generality, that belongs to the surface $P$.  Applying Proposition \ref{Mainth} we have

\begin{equation}\label{eqmain}
\begin{aligned}
-2\pi\chi &(D_t) +(\frac{\alpha}{2 t^2})\Vol(D_t)\\&+(\frac{1}{t}-\frac{\alpha}{2t})\int_{\partial D_t} \frac{1}{\Vert \nabla^P r\Vert} \leq \frac{1}{2}R(t)+\frac{t}{2\alpha}R'(t)
\end{aligned}
\end{equation}

Now, as  $\int_P ||B^P||^2 d\sigma < \infty$, we can apply  Proposition 2.2 in \cite{A1}, so we have, for $\alpha \in ]0,2[$, 
\begin{equation}
\frac{t}{2\alpha}R'(t)=\frac{t}{2\alpha}\int_{\partial D_t} \frac{\Vert B^P\Vert^2}{\Vert\nabla^Pr\Vert}d\sigma \leq\frac{\mu(t)}{2\alpha t} \int_{\partial D_t} \frac{1}{\Vert\nabla^Pr\Vert}d\sigma
\end{equation}
being $\mu(t)$ such that $\lim_{t \to \infty} \mu(t)=0$
and therefore, from (\ref{eqmain}),

\begin{equation}\label{eqAfterAnderson}
\begin{aligned}
&-2\pi\chi (D_t) +\Vol(D_t)(\frac{\alpha}{2t^2})\\&+(\frac{1}{t}-\frac{\alpha}{2t}-\frac{\mu(t)}{2\alpha t})\int_{\partial D_t} \frac{1}{\Vert\nabla^P r\Vert} d\sigma_t\, \leq  \,\frac{1}{2}R(t)
\end{aligned}
\end{equation}

On the other hand, $ \frac{1}{t}-\frac{\alpha}{2t}-\frac{\mu(t)}{2\alpha t} \geq 0$ if and only if $\mu(t) \leq \alpha(2-\alpha)$, which it is true for $t$ big enough, namely, for $t > t_\alpha$ because $\lim_{t \to \infty} \mu(t)=0$. Hence, as $\Vol(\partial D_t) \leq \int_{\partial D_t} \frac{1}{\Vert\nabla^P r\Vert} d\sigma_t$, and applying Theorem 1.1 in \cite{Pa}, we have that inequality (\ref{eqAfterAnderson}) becomes, for all $t > t_\alpha$

\begin{equation}\label{m}
\begin{aligned}
&-2\pi\chi (D_t) \\&+\frac{\Vol(D_t)}{\Vol(B_t^{0,2})}\left[2\pi(1-\frac{\alpha}{2}-\frac{\mu(t)}{2\alpha})+  \frac{\pi\alpha}{2} \right] \leq\, \frac{1}{2}R(t)
\end{aligned}
\end{equation}

Then, taking limits when $t \to \infty$ in inequality (\ref{m}) and applying Theorem \ref{CorDifeo}, we have that $\lim_{t \to \infty}\mu(t)=0$ and $\chi(P)=\lim_{t \to \infty}\chi(D_{t})$, so we obtain, for all $\alpha$ such that $0<\alpha<2$:
\begin{equation}\label{i}
\begin{aligned}
2\pi \Sup_t\frac{\Vol(D_{t})}{\Vol(B_{t}^{0,2})}&\left(1-\frac{\alpha}{2}+\frac{\pi\alpha}{2}\right)\\&-2\pi\chi(P)\leq \frac{\int_P \Vert B^P \Vert ^2}{2}<\infty
\end{aligned}
\end{equation}
Therefore we obtain $\Sup_{t>0}\frac{\Vol(D_t)}{\Vol(B_t^{0,2})} < \infty$ and $-\chi(P)<\infty$.

Then, letting $\alpha$ tend to $0$ we obtain, for all $t>0$:
\begin{equation}
\begin{aligned}
\Sup_{t>0}\frac{\Vol(D_t)}{\Vol(B_t^{0,2})}-\chi(P)\leq \frac{\int_P \Vert B^P \Vert ^2}{4\pi}
\end{aligned} 
\end{equation}

\section{Proof of Theorem \ref{ChernEq}}

In Corollary \ref{CorDifeo}, it was obtained a  sufficienty large  radius $R_0$, such that the extrinsic distance $r_p$ has no critical points in $P\setminus \D_{R_0}$.

Hence for this sufficienty large fixed radius $R_0$, there is a diffeomorphism

$$
\Phi: P\setminus \D_{R_0} \to \partial D_{R_0} \times [0,\infty[
$$

so, in particular, $P$ has only finitely many ends, each of finite topological type.

The above diffeomorfism implied that we could construct $P$ from $D_{R_0}$ ($R_0$ big enough) attaching annulis and that $\chi(P\setminus D_{t})=0$ when $t\geq R_0$, and hence for all $t>R_0$, 
$\chi(P)=\chi(D_{t})$.

Let us consider now an exhaustion by extrinsic balls $\{D_t\}_{t>0}$ of $P$ such that the extrinsic distance $r_o$ has no critical points in $P\setminus \D_{R_0}$.

Applying now Gauss-Bonnet Theorem to the extrinsic balls $D_{t}$

\begin{equation}
2\pi\chi(P)=\int_{D_t}K^Pd\sigma +\int_{\partial D_t}k_gd\sigma_t
\end{equation}

Having in to account equation (\ref{equ:kg}) and the Gauss formula, we have, for all sufficiently large radius $t > R_0$

\begin{equation}\label{P}
\begin{aligned}
2\pi\chi(P)&=-\frac{1}{2}\int_{D_t}\Vert B^P\Vert^2+b \Vol(D_t)+ h_b(t) \left(\Vol(D_t)\right)'\\&+\int_{\partial D_t} \langle B^P(e,e),\frac{\gr^\perp r}{\Vert \gr^P r \Vert} \rangle d\sigma_t
=-\frac{1}{2}\int_{D_t}\Vert B^P\Vert^2d\sigma\\&
+\frac{\Vol(D_t)}{\Vol(B_t^{b,2})}\left(b\cdot \Vol(B_t^{b,2})+h_b(t)(\Vol(D_t))' \frac{\Vol(B_t^{b,2})}{\Vol(D_t)}\right.\\&\left.+\frac{\Vol(B_t^{b,2})}{\Vol(D_t)}\int_{\partial D_t} \langle B^P(e,e),\frac{\gr^\perp r}{\Vert \gr^P r \Vert}\rangle d\sigma_t\right)
\end{aligned}
\end{equation}

But $2\pi= b\cdot \Vol(B_t^{b,2})+ h_b(t)\Vol(S^{b,1}_t)\,\,\,\forall t>0$, so, for all sufficiently large radius $t > R_0$ and after some computations:

\begin{equation}\label{A}
\begin{aligned}
2\pi\chi(P)=-\frac{1}{2}\int_{D_t}\Vert B^P\Vert^2d\sigma&+2\pi \frac{\Vol(D_t)}{\Vol(B_t^{b,2})}+h_b(t)\Vol(B^{b,2}_t)(\frac{(\Vol(D_t))}{\Vol(B^{b,2}_t)})' \\&+\int_{\partial D_t} <B^P(e,e),\frac{\gr^\perp r}{\Vert \gr^P r \Vert}>d\sigma_t
\end{aligned}
\end{equation}

The above equation is valid for all $t>R_0$, so, taking limits when $t \to \infty$, we can define
\begin{equation}
\begin{aligned}
G_b(P)&:= \lim_{t \to \infty}\left(h_b(t)\Vol(B^{b,2}_t)(\frac{(\Vol(D_t))}{\Vol(B^{b,2}_t)})' \right.\\&\left.+\int_{\partial D_t} <B^P(e,e),\frac{\gr^\perp r}{\Vert \gr^P r \Vert}>d\sigma_t)\right)
\end{aligned}
\end{equation}

Using equalities (\ref{A}), we have that
\begin{equation}
G_b(P)=2\pi\chi(P)+\frac{1}{2}\int_{D_t}\Vert B^P\Vert^2d\sigma-2\pi\Sup_{t}\frac{\Vol(D_t)}{\Vol(B_t^{b,2})} <\infty
\end{equation}
and hence, $G_b(P)$ do not depends on the exhaustion $\{D_{t}\}_{t>0}$.

\end{document}